\documentclass[11pt]{article}

\usepackage[all,cmtip]{xy}
\usepackage{graphicx}
\usepackage{amsmath}
\usepackage{amsfonts}
\usepackage{amssymb}
\usepackage{amsthm}
\usepackage{amscd}
\usepackage{textcomp}
\usepackage{hyperref}

\newcommand{\cB}{\mathcal{B}}
\newcommand{\Adual}{A^{\vee}}

\DeclareFontEncoding{OT2}{}{} 
  \newcommand{\textcyr}[1]{%
    {\fontencoding{OT2}\fontfamily{wncyr}\fontseries{m}\fontshape{n}%
     \selectfont #1}}
\newcommand{\Sha}{{\mbox{\textcyr{Sh}}}}

\newcommand{\cA}{\mathcal{A}}

\newcommand{\hra}{\hookrightarrow}
\newcommand{\ra}{\rightarrow}
\newcommand{\lra}{\longrightarrow}

\newcommand{\tensor}{\otimes}

\newcommand{\ssstyle}{\scriptscriptstyle}

\newcommand{\Q}{\mathbf{Q}}

\newcommand{\G}{\mathbb{G}}

\newcommand{\Z}{\mathbf{Z}}

\newcommand{\kbar}{\overline{k}}

\newcommand{\isom}{\cong}

\newcommand{\m}{\mathfrak{m}}

\renewcommand{\O}{\mathcal{O}}

\DeclareMathOperator{\coker}{coker}
\DeclareMathOperator{\Br}{Br}

\DeclareMathOperator{\Gal}{Gal}
\DeclareMathOperator{\im}{im}

\DeclareMathOperator{\Norm}{Nm}

\DeclareMathOperator{\Pic}{Pic}

\DeclareMathOperator{\Vis}{Vis}
\DeclareMathOperator{\Res}{Res}

\DeclareMathOperator{\tor}{tor}
\DeclareMathOperator{\Spec}{Spec}

\DeclareMathOperator{\ur}{ur}

\DeclareMathOperator{\Hom}{Hom}
\DeclareMathOperator{\Sel}{Sel}
\DeclareMathOperator{\Ext}{Ext}

\theoremstyle{plain}
\newtheorem{theorem}{Theorem}[section]
\newtheorem{proposition}[theorem]{Proposition}
\newtheorem{corollary}[theorem]{Corollary}

\newtheorem{lemma}[theorem]{Lemma}

\theoremstyle{definition}

\theoremstyle{remark}

\newtheorem{remark}[theorem]{Remark}

\numberwithin{equation}{section}

 1

\newcommand{\thetitle}
{Tamagawa Torsors of an Abelian Variety}

\begin{document}

\title{\thetitle}
\author{Saikat Biswas}
\maketitle

\begin{abstract}
For an abelian variety $A$ over a number field $K$, we define the set of \emph{Tamagawa torsors} of $A$ at a prime $v$ of $K$ to be the set of principal homogeneous spaces of $A$ over the completion $K_v$ of $K$ at $v$ that are split by an unramified extension of $K_v$. In this paper, we study the arithmetic properties of the Tamagawa torsors. We also give, following Mazur's theory of visibility, conditions under which
non-trivial elements of the Tamagawa torsors of $A$ may be interpreted as rational points on another abelian variety $B$.

\end{abstract}

\section*{Introduction}\label{intro}
Let $A$ be an abelian variety defined over a number field $K$. The Shafarevich-Tate group of $A$, denoted by $\Sha(A/K)$, is the set of isomorphism classes of principal homogeneous spaces of $A$ (also called $A$-torsors) defined over $K$ that are split by every completion $K_v $ of $K$, where $v$ is a prime of $K$ (\cite{silverman:AEC}). In other words, the set of non-trivial $A$-torsors\footnote{The \emph{trivial} $A$-torsors are isomorphic to $A$ over $K$} in $\Sha(A/K)$ have a $K_v$-rational point for every prime $v$ but do not have any $K$-rational points. The group $\Sha(A/K)$ is a fundamental arithmetic invariant of $A/K$ and its conjectural finiteness can be potentially used in effectively determining the set of rational points $A(K)$. In this paper, we define a new arithmetic invariant of $A$. Specifically, we consider the set of $A$-torsors defined over $K_v$ that are split by an unramified extension of $K_v$. We call this the set of \emph{Tamagawa torsors} of $A$ at $v$ and denote it by $TT(A/K_v)$. The terminology is rather prosaically explained by the fact (proved in Section 2) that $TT(A/K_v)$ is a finite set of order $c_{\ssstyle{A},v}$, the usual \emph{Tamagawa number} of $A$ at $v$ (\cite{silverman:AEC}).\footnote{We do welcome suggestions for a different terminology and/or notation.} As we explain presently, this is a useful way of interpreting the Tamagawa number of $A$ at $v$, particularly in the context of the Birch and Swinnerton-Dyer (BSD) Conjecture. In this paper, we study some of the arithmetic properties of $TT(A/K_v)$ including its duality properties as well as its relationship to other invariants of $A/K$ such as the Selmer group of $A$ and $\Sha(A/K)$. We also show how to construct non-trivial elements of $TT(A/K_v)$ via the rational points of another variety $B$ defined over $K$.\\Tamagawa torsors correspond to locally unramified cohomology classes and these have been fairly well-studied in the literature. Consequently, most of the theorems discussed in this paper are simple reinterpretations of known results. In particular, the idea of interpreting the Tamagawa number in terms of torsors is evident in \cite{mazur:towers} as well as \cite{milne:duality} and directly stated in \cite{stein:tamagawa}. The splitting property of these torsors is mentioned explicitly in \cite{gonzalez}. Our aim has been to unify the somewhat diverse results by interpreting the Tamagawa number in terms of principal homogeneous spaces.\\
We now briely describe the contents and organization of this paper. In Section 1, we present some definitions and results involving the N{\'e}ron model of $A$. In Section 2, we define Tamagawa torsors and prove its basic finiteness property. In Section 3, we study the duality properties of Tamagawa torsors. In Section 4, we relate the Tamagawa torsors to the Brauer group of the field over which it is defined, and consequently associate a Severi-Brauer variety as well as a Galois extension to a Tamagawa torsor. In Section 5, we relate Tamagawa torsors to Selmer groups. Finally, in Section 6, we study visibility properties of Tamagawa torsors. We first show that every Tamagawa torsor is `visible' in some ambient variety. We then prove an extension of a theorem of Mazur by means of which Tamagawa torsors of $A$ can be interpreted as rational points on another abelian variety $B$ also defined over $K$.

\noindent {\it Acknowledgements:} 
I thank Douglas Ulmer for his helpful comments and suggestions on an earlier draft. I also thank Dino Lorenzini for explaining some technical details pertaining to some of the results in this paper, and also for directing me to the relevant literature. Finally, I thank Amod Agashe, my advisor, for his encouragement and support.

\section{Basic Definitions}
Let $K$ be a number field with ring of integers $\O_K$ and let $A/K$ be an abelian variety over $K$. Let $\cA$ 
denote the \emph{N{\'e}ron model} of $A/K$ over $X=\Spec{\O_K}$ (\cite{neronmodels}). Thus $\cA$ is separated and of finite type over $X$ with generic fiber $A/K$, and satisfies the \emph{N{\'e}ron mapping property}: for each smooth $X$-scheme $S$ with generic fiber $S_{K}$, the restriction map $${\Hom_{X}(S,{\cA})\ra\Hom_{K}(S_{K},A)}$$ is bijective. Alternatively, recalling (\cite[\S II.1]{milne:etale}) that $A$ defines a sheaf (which we also denote as $A$) for the {\'e}tale (or the flat {\it{fpqf}}) topology over $\Spec{K}$, we consider the direct-image sheaf $j_{*}(A)$ for the {\'e}tale (or {\it{fpqf}}) topology over $X=\Spec{\O_K}$, where $j:\Spec{K}\hra{X}$ is the inclusion of the generic point. When the sheaf $j_{*}(A)$ is \emph{representable}, the smooth scheme $\cA\to X$ representing $j_{*}(A)$ will be called a N{\'e}ron model of $A$.\footnote{It is a deep theorem of N{\'e}ron that $\cA$ exists, up to isomorphism} Abusing notation, we identify the functor $j_{*}(A)$ itself as the N{\'e}ron model of $A$ and write $\cA=j_{*}(A)$. 
Over the flat topology for $X$, we have a short exact sequence of group schemes (or sheaves): 
\begin{equation}\label{SES}
0\to \cA^{0} \to {\cA} \to {\Phi}_A \to {0}
\end{equation}
where $\cA^{0}$ is the largest open subgroup scheme of $\cA$ with connected fibers (also called the \emph{identity component})and ${\Phi}_A=\cA/\cA^{0}$ is the \emph{component group} of $A$. If we regard $\Phi_A$ as an {\'e}tale sheaf over $X$ and denote by $\Phi_{A,v}$ its stalk at a prime $v$, then $\Phi_{A,v}$ can be considered as a finite, {\'e}tale group scheme over $\Spec{k_v}$. Equivalently, $\Phi_{A,v}$ is a finite abelian group equipped with an action of $\Gal({\kbar_v}/{k_v})$, where $k_v$ is the residue field of $K_v$ . Over $\Spec{k_v}$, we thus have an exact sequence of group schemes 
\begin{equation}\nonumber
0 \to \cA_v^0 \to \cA_v \to \Phi_{A,v} \to 0
\end{equation}
The group scheme $\Phi_{A,v}={\cA_{v}}/{\cA_{v}^{0}}$ of connected components is called the \emph{component group} of $\cA$ at $v$ and $c_{\ssstyle{A},v}=\#{\Phi_{A,v}(k_v)}$ is called the \emph{Tamagawa number} of $A$ at $v$. Considering the natural closed immersion $i_{v}:\Spec{k_v}\hra{X}$, we find that 
\begin{equation}\label{compdef}
\Phi_A=\bigoplus_{v}(i_{v})_{*}{{\Phi}_{A,v}}
\end{equation}
where the direct sum is taken over all $v$ or equivalently, over the finite set of $v$ where $A$ has bad reduction.
The short exact sequence (\ref{SES}) of {\'e}tale (or, flat) sheaves over $X$ induces a long exact sequence of {\'e}tale (or flat) cohomology groups:
\begin{equation}\label{LES}
0 \to \cA^{0}(X) \to \cA(X) \to \Phi_A(X)
\to H^1(X,\cA^0) \to H^1(X,\cA) \to H^1(X,\Phi_A)
\end{equation}
where we can write, for all $i$,
\begin{equation}\label{component}
H^i(X,\Phi_A)=\bigoplus_{v}H^i(\Spec{k_v},\Phi_{A,v})
\end{equation}
which follows from (\ref{compdef}). The N{\'e}ron mapping property implies that $$\cA(X)\isom A(K)$$ 
Furthermore, it is also known (see the Appendix to \cite{mazur:towers}) that, staying away from $2$-primary components, 
\begin{equation} \nonumber \label{def:sha}
\Sha(A/K)\isom\im\left[H^1(X,\cA^{0})\lra H^1(X,\cA)\right]
\end{equation}
where $\Sha(A/K)$ is the Shafarevich-Tate group of $A$.
Thus, the exact sequence (\ref{LES}) becomes
\begin{equation}\label{mainseq}
0 \to \cA^0(X) \to A(K) \to \bigoplus_{v}\Phi_{A,v}(k_v) \to H^1(X,\cA^0) \to \Sha(A/K) \to 0
\end{equation}
In particular, staying away from $2$-primary components, we note that the group $\Sha(A/K)$ may be expressed in two ways. 
First, the sequence 
\begin{equation}
\bigoplus_{v}\Phi_{A,v}(k_v) \to H^1(X,\cA^0) \to \Sha(A/K) \to 0
\end{equation}
identifies $\Sha(A/K)$ as a cokernel. Secondly, the sequence 
\begin{equation}\label{kernelsha}
0 \to \Sha(A/K) \to H^1(X,\cA) \to \bigoplus_{v}H^1(\Spec{k_v},\Phi_{A,v}) 
\end{equation}
identifies $\Sha(A/K)$ as a kernel.

\section{Tamagawa Torsors}

Let $A/K_v$ be the abelian variety defined over the completion $K_v$ of $K$ at a prime $v$. Let $K_v^{\ur}$ be the
maximal unramified extension of $K_v$. The inclusion $\Gal({\overline{K_v}}/K_v^{\ur})\subset\Gal({\overline{K_v}}/K_v)$
induces a map
$$H^1(K_v,A) \lra H^1(K_v^{\ur},A)$$
whose kernel corresponds to the \emph{unramified} subgroup of $H^1(K_v,A)$.
The map may also be given as 
$$WC(A/K_v) \lra WC(A/K_v^{\ur})$$
where $WC(A/K_v)\isom H^1(K_v,A)$ denotes the Weil-Ch{\^a}telet group of $A$ over $K_v$. In this situation, the kernel corresponds to the set of principal homogeneous spaces of $A$(or $A$-torsors) defined over $K_v$ that are split by $K_v^{\ur}$. Since $K_v^{\ur}$ is the directed union of finite unramified Galois extensions of $K_v$, it follows that the kernel corresponds to non-trivial $A$-torsors over $K_v$ that do not have a $K_v$-rational point but have a $L_i$-rational point for some finite unramified Galois extension $L_i$ of $K_v$. We denote this set by $TT(A/K_v)$ and call it the set of \emph{Tamagawa torsors} of $A$ at $v$. In particular, $TT(A/K_v)$ is the subgroup of unramified cohomology classes in $H^1(K_v,A)$.
In order to analyze the group $TT(A/K_v)$, we begin with the following proposition.

\begin{proposition}\label{ttprop}
Let $L_i$ be a finite, unramified Galois extension of $K_v$ with residue field $l_i$, and let $A/K_v$ be an abelian variety over $K_v$. Then
$$H^1(L_i/K_v,A(L_i))\isom H^1(l_i/k_v,\Phi_{A,v}(l_i))$$
\end{proposition}

\begin{proof}
Let $\O_{L_i}$ be the ring of integers of $L_i$ and $\m_{L_i}$ its maximal ideal so that $l_i=\O_{L_i}/{\m_{L_i}}$. In particular, $l_i$ is a finite Galois extension of the finite field $k_v$ and there is a canonical isomorphism
$\Gal(L_i/K_v)\isom\Gal(l_i/k_v)$. Let us denote this group as $G_i$. Let $\cA_{X_v}=\cA{\times}_{X}X_v$, where $X_v=\Spec{\O_{K_v}}$.
In particular, $\cA_{X_v}$ is the N{\'e}ron model of $A/K_v$ over $X_v$. Since $\cA_{X_v}$ is smooth over $X_v$, it follows that the reduction map 
$$\cA_{X_v}(\O_{L_i})\to\cA_{v}(l_i)$$
is surjective (\cite{milne:etale},I \S 4.13). Thus there is an exact sequence of $G_i$-modules
$$0 \to \cA^0_{X_v}(\O_{L_i}) \to \cA_{X_v}(\O_{L_i}) \to \Phi_{A,v}(l_i) \to 0$$
We also find that $\cA_{X_v}\tensor_{X_v}\Spec{\O_{L_i}}$ is the N{\'e}ron model of $A\tensor_{K_v}{L_i}$ so that in particular
$\cA_{X_v}(\O_{L_i})\isom A(L_i)$. The above sequence now becomes
$$0 \to \cA^0_{X_v}(\O_{L_i}) \to A(L_i) \to \Phi_{A,v}(l_i) \to 0$$
It therefore suffices to show that $H^n(G_i,\cA^0_{X_v}(\O_{L_i}))=0$ for $n=1,2$. Let $[\xi]\in H^1(G_i,\cA^0_{X_v}(\O_{L_i}))$ be represented by
an $\cA^0_{X_v}$-torsor $C$. Since $\cA_{v}^0$ is a connected algebraic group over the finite field $k_v$, it follows from Lang's Theorem that
the $\cA_{v}^0$-torsor $C\tensor_{\O_{K_v}}k_v$ is trivial so that $C(k_v)\neq\varnothing$. It now follows from Hensel's lemma that 
$C(\O_{K_v})\neq\varnothing$, and hence $[\xi]=0$.\\
On the other hand, we find that $H^2(G_i,\cA_{X_v}^0(\O_{L_i}/{\m_{L_i}^r}))=0$ for all $r$, since $G_i$ has cohomological dimension $1$. It follows (\cite{serre:localfields} XII, \S 3 Lemma 3) that $H^2(G_i,\cA_{X_v}^0(\O_{L_i}))=0$.
\end{proof}

\begin{theorem}\label{ttorder}
The set $TT(A/K_v)$ is finite and $\#TT(A/K_v)=c_{A,v}$.
\end{theorem}

\begin{proof}
The inflation-restriction sequence
$$0 \lra H^1(K_v^{\textrm{ur}}/K_v,A(K_v^{\textrm{ur}})) \lra H^1(K_v,A) \lra H^1(K_v^{\textrm{ur}},A)$$
identifies the set of Tamagawa torsors with the injective image of the group $H^1(K_v^{\textrm{ur}}/K_v,A(K_v^{\textrm{ur}}))$ in $H^1(K_v,A)$. 
There is an isomorphism
$$H^1(K_v^{\textrm{ur}}/K_v,A(K_v^{\textrm{ur}}))\isom\varinjlim{H^1(L_i/K_v,A(L_i))}$$
where the direct limit is over all finite, unramified Galois extensions $L_i$ of $K_v$. It now follows from Proposition \ref{ttprop} above
that 
$$H^1(K_v^{\ur}/K_v,A(K_v^{\ur}))\isom\varinjlim{H^1(l_i/k_v,\Phi_{A,v}(l_i))}\isom H^1(k_v,\Phi_{A,v})$$
But $\Phi_{A,v}$ is a finite Galois module over the finite field $k_v$ so that its Herbrand quotient is $1$. This implies that
$$\#H^1(\Spec{k_v},\Phi_{A,v})=\#H^0(\Spec{k_v},\Phi_{A,v})=c_{A,v}$$
\end{proof}

\begin{remark}
The proof above shows that the set $TT(A/K_v)$ is trivial when $A$ has \emph{good reduction} at $v$ (cf. \cite{silverman:AEC}).
It also follows from the theorem that the direct sum $\bigoplus_{v}TT(A/K_v)$ is the set of Tamagawa torsors of $A$ over all $v$ and has order $\prod_{v}c_{A,v}$. 
\end{remark}

\section{Local Duality of Tamagawa Torsors}

Let $A^{\vee}$ be the abelian variety dual to $A$ over $K_v$. For any abelian group $M$, we denote by $M^{*}:=\Hom(M,\Q/\Z)$ the \emph{Pontryagin dual} of $M$ or the character group of $M$. The main theorem in this section identifies the character group of $TT(A/K_v)$ as the arithmetic component group $\Phi_{A^{\vee},v}(k_v)$.

\begin{theorem}\label{ttdual}
There is a canonical pairing $$\Phi_{A^{\vee},v}(k_v)\,\times\,TT(A/K_v) \to \Q/\Z$$ which induces an isomorphism $${TT(A/K_v)}^{*}\isom\Phi_{A^{\vee},v}(k_v)$$
\end{theorem}

Theorem \ref{ttdual} will follow as a consequence of two lemmas that we now establish. To begin with, there is a canonical, non-degenerate pairing due to Tate(\cite{milne:duality} Cor I.3.4, \cite{tate:wc}) 
$$A^{\vee}(K_v)\,\times\,H^1(K_v,A) \to \Q/\Z$$ 
which induces an isomorphism $A^{\vee}(K_v)\isom H^1(K_v,A)^{*}$ of locally compact groups. 

\begin{lemma}\label{lem1}
Let $L_i$ be a finite, unramified Galois extension of $K_v$. Then there is a perfect pairing
$$\Adual(K_v)/{\Norm(\Adual(L_i))}\,\times\,H^1(L_i/K_v,A(L_i)) \to \Q/\Z$$
where $\Norm(\Adual(L_i))\subset\Adual(K_v)$ is the image of $\Adual(L_i)$ in $\Adual(K_v)$ under the norm mapping
corresponding to $\Gal(L_i/K_v)$.
\end{lemma}

\begin{proof}
The restriction map 
$$H^1(K_v,A) \to H^1(L_i,A)$$
fits into a commutative diagram
\[
\xymatrix{
H^1(K_v,A)\ar[r]\ar[d] &H^1(L_i,A)\ar[d]\\
{\Adual(K_v)}^{*}\ar[r] &{\Adual(L_i)}^{*}}
\]
where the vertical arrows are isomorphisms by Tate's local
duality. The bottom horizontal map is dual (\cite{tate:wc}) 
to the norm map 
$$\Norm_{L_i/K_v}\;:\;\Adual(L_i)\to\Adual(K_v)$$
so that there are isomorphisms (\cite{tate:wc} 8, Cor 1)
\begin{align*}
H^1(L_i/K_v,A(L_i)) &\isom\ker\left(H^1(K_v,A)\to H^1(L_i,A)\right)\\
&\isom\ker\left({\Adual(K_v)}^{*}\to{\Adual(L_i)}^{*}\right)\\
&\isom\left(\coker\left(\Adual(L_i)\overset{\Norm}{\longrightarrow}\Adual(K_v)\right)\right)^{*}\\
&\isom(\Adual(K_v)/{\Norm(\Adual(L_i))})^{*}
\end{align*}
In particular, $\Norm(\Adual(L_i))\subset\Adual(K_v)$ is the exact annihilator of the subgroup 
$H^1(L_i/K_v,A(L_i))\subset H^1(K_v,A)$ under the Tate pairing described above.
\end{proof}

Since $H^1(L_i/K_v,A(L_i))\isom H^1(l_i/k_v,\Phi_{A,v})$ by Proposition \ref{ttprop} and since
the latter group is finite, we conclude that $\Adual(K_v)/{\Norm(\Adual(L_i))}$ is finite as well. Subgroups
of $\Adual(K_v)$ of the form $\Norm(\Adual(L_i))$ for any finite extension $L_i$ will be called the \emph{norm subgroups}
of $\Adual(K_v)$. We define the group of \emph{universal norms} on $\Adual(K_v)$ from $K_v^{\ur}$ to be
$$\Norm(\Adual(K_v^{\ur}))=\bigcap_{L_i}\Norm(\Adual(L_i))$$
where the intersection is over all finite, unramified extensions $L_i$ of $K_v$. It follows from Lemma \ref{lem1} that 
$TT(A/K_v)=\varinjlim_{L_i}H^1(L_i/K_v,A(L_i))$ is the Pontryagin dual of $\Adual(K_v)/\Norm(\Adual(K_v^{\ur}))$. To prove Theorem \ref{ttdual},
it therefore suffices to show that

\begin{lemma}
For an abelian variety $A/K_v$, the arithmetic component group $\Phi_{A,v}(k_v)$ is isomorphic to $A(K_v)$ modulo the subgroup of universal norms from
$K_v^{\ur}$, i.e.
$$\Phi_{A,v}(k_v)\isom A(K_v)/\Norm(A(K_v^{\ur}))$$
\end{lemma}

\begin{proof} 
The norm map $A(L_i)\to A(K_v)$ may be given as a map $$\cA_{X_v}(\O_{L_i})\to\cA_{X_v}(\O_{K_v})$$ where $\cA_{X_v}$ is as defined in the proof of Proposition \ref{ttprop}. Choosing the finite, unramified extension $L_i$ such that $[L_i:K_v]=[l_i:k_v]$ is coprime to $c_{\ssstyle{A},v}=\#{\Phi_{A,v}(k_v)}$, we find that there is an induced norm map 
$${\Norm}^0:\cA_{X_v}^0(\O_{L_i})\to\cA_{X_v}^0(\O_{K_v})$$
whose kernel $G=\ker(\Norm^0)$ is a smooth group scheme with connected fibers. Let $x:\Spec{\O_{K_v}}\to\cA_{X_v}^0$ be any $\O_{K_v}$-point of $\cA_{X_v}^0$. The $\Norm^0$-pullback of $x$ is a smooth $\O_{K_v}$-scheme $Y$ that is also a $G$-torsor for the {\'e}tale topology. The special fiber $Y_v=Y\times_{\Spec{\O_{K_v}}}\Spec{k_v}$ is a torsor for a smooth, connected group scheme over the finite field $k_v$. By Lang's Theorem, $Y_v$ has a $k_v$-rational point. Since $Y$ is smooth over $\O_{K_v}$, the reduction map $$Y(\O_{K_v})\to Y_{v}(k_v)$$ is surjective and thus, $Y$ has an $\O_{K_v}$-point. We conclude that $\Norm^0$ is surjective. Now consider the commutative diagram
\[
\xymatrix{
0\ar[r] &\cA_{X_v}^0(\O_{L_i})\ar[r]\ar[d]^{\Norm^0} &\cA_{X_v}(\O_{L_i})\isom A(L_i)\ar[r]\ar[d] &\Phi_{A,v}(l_i)\ar[r]\ar[d] &0\\
0\ar[r] &\cA_{X_v}^0(\O_{K_v})\ar[r] &\cA_{X_v}(\O_{K_v})\isom A(K_v)\ar[r] &\Phi_{A,v}(k_v)\ar[r] &0}
\]
The left vertical map is surjective while the right vertical map is the $0$ map. It follows that $A(K_v)/\bigcap_{L_i}\Norm(A(L_i)\isom\Phi_{A,v}(k_v)$.
\end{proof}


\begin{remark}
The pairing in Theorem \ref{ttdual}also follows directly from the canonical pairing $\Phi_{A^{\vee},v}\times\Phi_{A,v}\to\Q/\Z$ defined by Grothendieck (cf. \cite{milne:duality} III C.13).
\end{remark}

\begin{theorem}
There are exact sequences
\[
\xymatrix{
0 \ar[r] &{\cA^{\vee}}^0(X)\ar[r] &A^{\vee}(K)\ar[r] &\displaystyle\bigoplus_{v}\Phi_{A^{\vee},v}(k_v)\\
0 \ar[r] &\Sha(A/K)\ar[r] &H^1(X,\cA)\ar[r] &\displaystyle\bigoplus_{v}TT(A/K_v)}
\]
where $\cA^{\vee}$ is the N{\'e}ron model of $A^{\vee}$. Assuming that $\Sha(A/K)$ is finite, the images of the two-right hand maps are orthogonal complements of each other under the pairing described in Theorem \ref{ttdual}.
\end{theorem}

\begin{proof}
The first exact sequence is \ref{mainseq} applied to $A^{\vee}$ while the second exact sequence follows from \ref{kernelsha} and Theorem \ref{ttorder}. The second part of the theorem follows from the main result in \cite{gonzalez}.
\end{proof}

\begin{remark}
The second exact sequence in the theorem above also shows that the index $[H^1(X,\cA):\Sha(A/K)]$ divides $\prod_{v}c_{A,v}$, the product of the Tamagawa numbers of $A$.
\end{remark}

\begin{remark}
According to the discussion in \cite{gonzalez}, the cokernel of the map $\Adual(K)\to\bigoplus_{v}\Phi_{\Adual,v}(k_v)$ may be defined as
the \emph{N{\'e}ron class group} of $\Adual/K$ and denoted by $C_{\Adual,K}$. If the image of the map $H^1(X,\cA)\to\bigoplus_{v}TT(A/K_v)$ is
denoted by $C_{A,K}^{1}$, then the above theorem states that, assuming $\Sha(A/K)$ is finite, there is a perfect pairing
$$C_{\Adual,K} \times C_{A,K}^{1} \to \Q/\Z$$
It also follows that the short exact sequence 
$$0 \to C_{\Adual,K} \to H^1(X,{\cA^{\vee}}^0) \to \Sha(\Adual/K) \to 0$$
has, as its dual, the short exact sequence
$$0 \to \Sha(A/K) \to H^1(X,\cA) \to C_{A,K}^1 \to 0$$
\end{remark}

\section{Brauer Groups and Tamagawa Torsors}

The Brauer group of $K_v$, denoted by $\Br(K_v)$, is given by the cohomology group $H^2(K_v,\G_m)$. To relate the Tamagawa torsors of $A$ at $v$ to $\Br(K_v)$, we begin by recalling 
(\cite[\S I 3.1]{milne:duality}) the canonical isomorphism 
$$\Pic^0_{K_v}(A)\isom\Ext^1_{K_v}(A,\G_m)$$ 
where $\Pic^0_{K_v}(A)=A^{\vee}(K_v)$. There is a canonical pairing (\cite[\S 0.16]{milne:duality})\footnote{This pairing can be used to establish
Tate local duality mentioned in the proof of Theorem \ref{ttdual}.}
$$\Ext^1_{K_v}(A,\G_m)\times H^1(K_v,A) \to H^2(K_v,\G_m)\isom\Br(K_v)$$
Thus every element in $\Ext^1_{K_v}(A,\G_m)$ defines a map
$$H^1(K_v,A) \to \Br(K_v)$$
We wish to study the restriction of this map to $TT(A/K_v)$. To this end, let $B_{\alpha}\in\Ext^1_{K_v}(A,\G_m)$ be the semiabelian variety corresponding to $0\neq[\alpha]\in\Pic^0_{K_v}(A)$. We define the Tamagawa torsors of $B_{\alpha}$ at $K_v$, denoted by $TT(B_{\alpha}/K_v)$, in
a manner similar to that employed for $A$. In other words, we define $TT(B_{\alpha}/K_v)$ by the exactness of the sequence
$$0 \to TT(B_{\alpha}/K_v) \to H^1(K_v,B_{\alpha}) \to H^1(K_v^{\ur},B_{\alpha})$$

\begin{proposition}\label{ttbr1}
With notations as above, there is an exact sequence 
$$0 \to TT(B_{\alpha}/K_v) \to TT(A/K_v) \overset{\Delta_{\alpha,v}}{\lra} \Br(K_v)$$
\end{proposition}

\begin{proof}
We have an exact sequence
$$0 \to \G_m \to B_{\alpha} \to A \to 0$$
over $K_v$ that gives rise to the diagram
\[
\xymatrix{
&0\ar[d] &0\ar[d] &0\ar[d]\\
&H^1(K_v^{\ur}/K_v,B_{\alpha})\ar[d] &TT(A/K_v)\ar[d] &\Br(K_v^{\ur}/K_v)\ar[d]\\
0\ar[r] &H^1(K_v,B_{\alpha})\ar[r]\ar[d] &H^1(K_v,A)\ar[r]\ar[d] &H^2(K_v,\mathbb{G}_m)\isom\Br(K_v)\ar[d]\\
0\ar[r] &H^1(K_v^{\ur},B_{\alpha})\ar[r] &H^1(K_v^{\ur},A)\ar[r] &H^2(K_v^{\ur},\mathbb{G}_m)\isom\Br(K_v^{\ur})}
\]
with exact rows and columns. The exactness of the rows follow from the long exact sequence of cohomology associated
to the short exact sequence above (over both $K_v$ and $K_v^{\ur}$). The exactness of the columns follow from the inflation-restriction
sequence associated to the inclusion $K_v\subset K_v^{\ur}$. We also use the fact that $H^1(K_v,\G_m)=0=H^1(K_v^{\ur},\G_m)$. 
The kernels of the vertical maps yield an exact sequence
$$0 \to H^1(K_v^{\ur}/K_v,B_{\alpha}(K_v^{\ur})) \to TT(A/K_v) \overset{\Delta_{\alpha,v}}{\lra} \Br(K_v^{\ur}/K_v)$$
However, $\Br(K_v^{\ur})=0$ (\cite[\S X.7]{serre:localfields}) so that $\Br(K_v^{\ur}/K_v)\isom\Br(K_v)$. The desired exact sequence now
follows.
\end{proof}

\begin{remark}
The above proof shows in particular that $TT(B_{\alpha}/K_v)$ is finite and its order divides $c_{{\ssstyle{A}},v}$. We denote the
order of $TT(B_{\alpha}/K_v)$ by $c_{{\ssstyle{B_{\alpha}}},v}$.
\end{remark}

Thus we have an exact sequence
$$0 \to TT(B_{\alpha}/K_v) \to TT(A/K_v) \to \im{\Delta_{\alpha,v}} \to 0$$
where, $\im{\Delta_{\alpha,v}}$ is a finite subgroup of $\Br(K_v)$. We may interpret this subgroup in two ways. First, there is an isomorphism $$\Br(K_v)\isom\bigcup_{L_i}\left(\bigcup_{n}BS_{n}(L_i/K_v)\right)$$
where $L_i$ is any finite, unramified Galois extension of $K_v$ of degree $n$ and $BS_{n}(L/K_v)$ is the set of isomorphism classes of Severi-Brauer varieties defined over $K_v$ i.e. $K_v$-varieties $V$ which become isomorphic to $\mathbb{P}^{n-1}$ over $L$, i.e. 
$$V\tensor_{K_v}L\isom\mathbb{P}_{L}^{n-1}$$
The proposition thus shows that every Severi-Brauer variety arises from a Tamagawa torsor that depends upon the choice of $[\alpha]\in\Pic^0_{K_v}(A)$. On the other hand, we have 
$$\Br(K_v)=\bigcup_{L_i}\Br(L_i/K_v)$$
where, as before, the union is over all finite, unramified extension $L_i$ of $K_v$. Consider the isomorphisms
$$\Br(L_i/K_v)\isom{K_v^{\times}/\Norm(L_i^{\times})}\isom\Gal(L_i/K_v)$$
where the first isomorphism follows from the periodicity of the cohomology of cyclic groups, while the second from local class field theory.
Thus $\im{\Delta_{\alpha,v}}$ is a finite subgroup of $\Gal(L_i/K_v)$ for some $L_i$. In particular, there is a finite sub-extension $F_{\alpha}$ of $L_i$ such that $\im{\Delta_{\alpha,v}}\isom\Gal(L_i/F_{\alpha})$. The short exact sequence 
$$0 \to TT(B_{\alpha}/K_v) \to TT(A/K_v) \to \Gal(L_i/F_{\alpha})  \to 0$$
thus shows that every finite sub-extension of a finite, unramified extension of $K_v$ arises from an isomorphism class of Tamagawa torsors.

\begin{remark}
The global exact sequence relating $H^1(K,A)$ to $\Br(K)$ is similar to the local situation described in the proof of Proposition \ref{ttbr1}. We thus have a diagram
\[
\xymatrix{
&0\ar[d] &0\ar[d] &\\
&\Sha(C_{\alpha}/K)\ar[d] &\Sha(A/K)\ar[d] &0\ar[d]\\
0\ar[r] &H^1(K,C_{\alpha})\ar[r]\ar[d] &H^1(K,A)\ar[r]\ar[d] &\Br(K)\ar[d]\\
0\ar[r] &\displaystyle\bigoplus_{v}H^1(K_v,C_{\alpha})\ar[r] &\displaystyle\bigoplus_{v}H^1(K_v,A)\ar[r] &\displaystyle\bigoplus_{v}\Br(K_v)}
\]
In particular, we have $\Sha(C_{\alpha}/K)\isom\Sha(A/K)$.
\end{remark}

\section{Selmer Groups and Tamagawa Torsors}

In this section, we relate the Tamagawa torsors to Selmer groups. We begin with the local case. Consider the Kummer exact sequence
$$0 \to A(K_v)\otimes{\Z/n\Z} \to H^1(K_v,A[n]) \to H^1(K_v,A)[n] \to 0$$
over $K_v$ for any integer $n$. Passage to direct limits yields the exact sequence
$$0 \to A(K_v)\otimes{\Q/\Z} \to H^1(K_v,A_{\tor}) \to H^1(K_v,A) \to 0$$
Let $H^1_{TT}(K_v,A_{\tor})\subset H^1(K_v,A_{\tor})$ be the inverse image of $TT(A/K_v)\subset H^1(K_v,A)$ under the surjection
$H^1(K_v,A_{\tor}) \to H^1(K_v,A)$. We then have an exact sequence
$$0 \to A(K_v)\otimes{\Q/\Z} \to H^1_{TT}(K_v,A_{\tor}) \to TT(A/K_v) \to 0$$
Identifying the injective image of $A(K_v)\otimes{\Q/\Z}$ in $H^1(K_v,A_{\tor})$ (as well as in $H^1_{TT}(K_v,A_{\tor})$) as the
\emph{local Selmer group} $\Sel(A/K_v)$, we have thus proved that

\begin{proposition}
Let $H^1_{TT}(K_v,A_{\tor})$ be defined as above. Then there is an exact sequence
$$0 \to \Sel(A/K_v) \to H^1_{TT}(K_v,A_{\tor}) \to TT(A/K_v) \to 0$$
In particular, the index of $\Sel(A/K_v)$ in $H^1_{TT}(K_v,A_{\tor})$ is equal to the
Tamagawa number of $A$ at $v$.
\end{proposition}

The global Kummer exact sequence is related to the local one by means of the commutative diagram

\xymatrix{
0\ar[r] &A(K)\tensor{\Z/n\Z}\ar[r]\ar[d] &H^1(K,A[n])\ar[r]\ar[d]\ar[rd] &H^1(K,A)[n]\ar[r]\ar[d] &0\\
0\ar[r] &\displaystyle\prod_{v}A(K_v)\tensor{\Z/n\Z}\ar[r] &\displaystyle\prod_{v}H^1(K_v,A[n])\ar[r] 
&\displaystyle\prod_{v}H^1(K_v,A)[n]\ar[r] &0}

where the vertical arrows are induced by the inclusions $\Gal(\overline{K}/K)\subset\Gal(\overline{K_v}/K_v)$ and 
$A(\overline{K})\subset A(\overline{K_v})$ 
for every $v$. The \emph{global $n$-Selmer group}, denoted by $\Sel_n(A/K)$, is the kernel of the diagonal map in the above diagram
so that there is an exact sequence
$$0 \to \Sel_{n}(A/K) \to H^1(K,A[n]) \to \displaystyle\prod_{v}H^1(K_v,A)[n]$$
Let $H^1_{TT}(K,A[n])$ be the subgroup of $H^1(K,A[n])$ that maps to $\bigoplus_{v}TT(A/K_v)[n]\subseteq\prod_{v}H^1(K_v,A)[n]$ under the map above.
There is thus an exact sequence
$$0 \to H^1_{TT}(K,A[n]) \to H^1(K,A[n]) \to \displaystyle\prod_{v}H^1(K_v^{\ur},A)[n]$$

\begin{proposition}
There is an exact sequence
$$0 \to \Sel_{n}(A/K) \to H^1_{TT}(K,A[n]) \to \displaystyle\bigoplus_{v}TT(A/K_v)[n]$$
\end{proposition}

\begin{proof}
Apply the kernel-cokernel exact sequence (\cite{milne:duality} I 0.24) to the pair of maps
$$H^1(K,A[n]) \to \displaystyle\prod_{v}H^1(K_v,A)[n] \to \displaystyle\prod_{v}H^1(K_v^{\ur},A)[n]$$
\end{proof}

Let $H^1_{TT}(K,A_{\tor})$ be the direct sum of the groups $H^1_{TT}(K,A[n])$ and the Selmer group $\Sel(A/K)$ be that of the groups $\Sel_{n}(A/K)$ over all $n$. We find that

\begin{corollary}
The index of the Selmer group $\Sel(A/K)$ in $H^1_{TT}(K,A_{\tor})$ divides $\displaystyle\prod_{v}c_{\ssstyle{A},v}$.
\end{corollary}

\begin{proof}
Pass on to the direct limit of the sequence in the proposition.
\end{proof}

\begin{remark}
The results proved so far show that there is a commutative diagram
\[
\xymatrix{
& &0\ar[d] &0\ar[d] &\\
0\ar[r] &A(K)\tensor{\Q/\Z}\ar[r]\ar@{=}[d] &\Sel(A/K)\ar[r]\ar[d] &\Sha(A/K)\ar[r]\ar[d] &0\\
0\ar[r] &A(K)\tensor{\Q/\Z}\ar[r] &H^1_{TT}(K,A_{\tor})\ar[r]\ar[d] &H^1_{\ur}(K,A)\ar[r]\ar[d] &0\\
& &\displaystyle\bigoplus_{v}TT(A/K_v)\ar@{=}[r] &\displaystyle\bigoplus_{v}TT(A/K_v) &}
\]
\end{remark}
that relates the Mordell-Weil group, the Selmer group, the Shafarevich-Tate group and the Tamagawa torsors of $A$. Here
$H^1_{\ur}(K,A)\isom H^1(X,\cA)$ is the subgroup of $H^1(K,A)$ that maps to $H^1(K_v^{\ur}/K_v,A(K_v^{\ur}))\subseteq H^1(K_v,A)$
for every $v$. In particular, the index of $\Sel(A/K)$ in $H^1_{TT}(K,A_{\tor})$ as well as that of $\Sha(A/K)$ in $H^1_{\ur}(K,A)$
divide $\prod_{v}c_{\ssstyle{A},v}$, the product of the Tamagawa numbers of $A$.

\section{Visibility of Tamagawa Torsors}

Let $\iota:A\hra J$ be an embedding\footnote{i.e. a morphism that is also a closed immersion} of abelian varieties over $K_v$. The kernel of the induced map $H^1(K_v,A)\to H^1(K_v,J)$ may be 
defined, following \cite{cremona-mazur}, as the \emph{visible} subgroup of $H^1(K_v,A)$ with respect to the embedding $\iota$ and denoted 
by $\Vis_{J}H^1(K_v,A)$\footnote{Although the visible subgroup depends on the choice of embedding, it is usually clear from the context
and is therefore omitted from the notation} i.e.
$$\Vis_{J}H^1(K_v,A):=\ker(H^1(K_v,A)\to H^1(K_v,J))$$
The terminology can be explained by noting that if $C=J/A$ then, over $K_v$, there is a short exact sequence
$$0 \to A \to J \overset{\pi}{\to} C \to 0$$
of $\Gal(\overline{K_v}/K_v)$-modules which induces a long exact sequence of cohomology groups
$$0 \to A(K_v) \to J(K_v) \to C(K_v) \to H^1(K_v,A) \to H^1(K_v,J) \to \cdots$$
which can be truncated to the exact sequence
$$0 \to J(K_v)/A(K_v) \to C(K_v) \to \Vis_{J}H^1(K_v,A) \to 0$$ 
Let $\xi\in\Vis_{J}H^1(K_v,A)$ be the image of $P\in C(K_v)$. Then $\pi^{-1}(P)$ is a coset of $A$ in $J$, 
and thus is a torsor under $A$. This explains how elements in $\Vis_{J}H^1(K_v,A)$ are `visible' in $J(\overline{K_v})$.\footnote{Considering
the embedding $A\hra J$ over the number field $K$, it can be shown (see \cite{cremona-mazur}) that $\Vis_{J}H^1(K,A)$ is finite}
Clearly we can define $\Vis_{J}TT(A/K_v)$, the visible part of the Tamagawa torsors, as 
$$\Vis_{J}TT(A/K_v):=\Vis_{J}H^1(K_v,A)\cap TT(A/K_v)$$
Let $L_i$ be a finite, unramified Galois extension of $K_v$ and let 
\mbox{$\Res_{L_i/K_v}(A_{L_i})$} be the \emph{restriction of scalars} of $A_{L_i}$ from $L_i$ to $K_v$.
In particular, it is an abelian variety over $K_v$ of dimension $[L_i:K_v].\dim(A)$

\begin{proposition}
Every element in $TT(A/K_v)$ is visible in $\Res_{L_i/K_v}(A_{L_i})$ for some finite, unramified Galois extension
$L_i$ of $K_v$.
\end{proposition}

\begin{proof}
There is a canonical embedding $A \hra \Res_{L_i/K_v}(A_{L_i})$ of abelian varieties over $K_v$ 
which induces a map 
$$H^1(K_v,A) \to H^1(K_v,\Res_{L_i/K_v}(A_{L_i}))$$ 
We thus have an exact sequence 
$$0 \to \Vis_{\Res_{L_i/K_v}(A_{L_i})}(H^1(K_v,A)) \to H^1(K_v,A) \to H^1(K_v,\Res_{L_i/K_v}(A_{L_i}))$$
On the other hand, the inflation-restriction sequence with respect to the extension $L_i/K_v$ is 
$$0 \to H^1(L_i/K_v,A(L_i)) \to H^1(K_v,A) \to H^1(L_i,A)$$
A straightforward application of Shapiro's lemma (\cite{serre:galois}) implies that there is an isomorphism 
$$H^1(K_v,\Res_{L_i/K_v}(A_{L_i})) \isom H^1(L_i,A)$$
It follows that we have an isomorphism 
$$\Vis_{\Res_{L_i/K_v}(A_{L_i})}(H^1(K_v,A)) \isom H^1(L_i/K_v,A(L_i))$$
Upon passage to the direct limit over such $L_i$s, we obtain isomorphisms
\begin{align*}
\varinjlim\Vis_{\Res_{L_i/K_v}(A_{L_i})}(H^1(K_v,A)) &\isom \varinjlim H^1(L_i/K_v,A(L_i))\\ 
&\isom H^1(K_v^{\ur}/K_v,A(K_v^{\ur}))\\ 
&\isom TT(A/K_v)
\end{align*}
where the last isomorphism follows from the proof of Theorem \ref{ttorder}.
\end{proof}

Having described the ambient variety (i.e. $\Res_{L_i/K_v}(A_{L_i})$) in which any Tamagawa torsor is always visible, we now give a method 
by means of which Tamagawa torsors of $A$ may be interpreted as $K$-rational points on another variety $B$ with which it shares a certain $p$-congruence. The following is a variation of the main theorem proved in \cite{agashe-biswas}

\begin{theorem}
Let $A$ and $B$ be abelian varieties of the same dimension over a number field $K$, having ranks $r_{\ssstyle{A}}=0$ and $r_{\ssstyle{B}}>0$ respectively and such that $B$ has semistable reduction over $K$. Let $N$ be an integer divisible by the residue characteristics of the primes of bad reduction for both $A$ and $B$. Let $p$ be an odd prime such that $e_p<p-1$, where $e_p$ is the largest ramification index of any prime of $K$ lying over $p$, and such that
$$gcd\left(p,\;N\cdot\;\#A(K)_{\tor}\cdot\;\prod_{v}c_{\scriptscriptstyle{B},v}\right)=1$$
Suppose further that $B[p]\isom A[p]$ over $K$. Assuming that $\Sha(A/K)$ has trivial $p$-primary components, there is an injection
$$B(K)/pB(K) \hra \bigoplus_{v}TT(A/K_v)[p]$$
\end{theorem}

\begin{proof}
We briefly sketch the proof, referring to \cite{agashe-biswas} for details. The isomorphism $A[p]\isom B[p]$ over $K$ induces an isomorphism $\cA[p]\isom\cB[p]$ over $X=\Spec{\O_K}$, where $\cA$ and $\cB$ are the corresponding N{\'e}ron models (this is the heart of the proof as given in \cite{agashe-biswas}). It then follows, given the conditions of the theorem, that we have a diagram 
 $$\xymatrix{
& 0\ar[d]\\
& \Sha(A/K)[p]\ar[d]\\
B(K)/pB(K)\ar[r]^{\varphi} & H^1(X,\cA)[p]\ar[d]\\
& \displaystyle\bigoplus_{v}H^1(\Spec{k_v},\Phi_{A,v})[p]}$$
such that $\ker(\varphi)=0$. However, $\Sha(A/K)[p]$ is trivial by the conditions of the theorem and $H^1(\Spec{k_v},\Phi_{A,v})[p]\isom TT(A/K_v)[p]$ (see the proof of Theorem \ref{ttorder}). The desired result now follows.
\end{proof}

We may also say, in the language of \cite{cremona-mazur}, that $TT(A/K_v)$ is `explained' by $B(K)$. We now reinterpret the example discussed in \cite{agashe-biswas}, which was first discovered in \cite{stein:tamagawa}. Consider the optimal elliptic curves $A$=114C1 and $B$=57A1. The data in \cite{cremona:algs} shows that $A$ has rank $0$, \mbox{$A(\Q)_{\tor}\isom{\Z/4\Z}$} and $\Sha(A/\Q)$ has trivial conjectural order. On the other hand, we find that $B(\Q)\isom\Z$ and $\prod_{l}c_{\ssstyle{B},l}=2$. Furthermore, we have $A[5]\isom B[5]$ over $\Q$. Thus, the triple $(A,B,5)$ satisfies the hypothesis in the above theorem and we conclude that there is an injection $\Z/5\Z\hra{\bigoplus_{p}TT(A/\Q_p)}$ i.e. $A$ has a Tamagawa torsor of order $5$. This agrees with the available data, according to which $\prod_{l}c_{\ssstyle{A},l}=20$.

\begin{remark}
For an optimal elliptic curve $A/\Q$ of rank $0$, the second part of the BSD Conjecture states that
$$\frac{L_{A,\Q}(1)}{\Omega_{A}}.(\#A(\Q)_{\tor})^{2}\overset{?}{=}\#\Sha(A/\Q).\prod_{p}c_{\ssstyle{A},p}$$
where $L_{A,\Q}(s)$ is the \emph{$L$-function} asociated to $A/\Q$, $\Omega_{A}$ is the \emph{real volume} of $A$ computed using a N{\'e}ron differential. Under the conditions of the above theorem, Agashe (\cite{agashe:visfac} Prop 1.5) has shown that $p$ divides the left-hand side of the BSD formula. Since $\Sha(A/\Q)$ is assumed to have trivial $p$-torsion, it follows that $p$ must divide the Tamagawa numbers of $A$. With our interpretation of the Tamagawa number as the number of Tamagawa torsors, this implies that $A$ must have a Tamagawa torsor of order $p$ which is precisely what the theorem confirms. Thus one may also view the above theorem as providing theoretical evidence for the BSD Conjecture.
\end{remark}

\newpage
\newcommand{\etalchar}[1]{$^{#1}$}


\begin{thebibliography}{FpS{\etalchar{+}}01}

\bibitem[AB13]{agashe-biswas}
A.~Agashe and S.~Biswas, \emph{Constructing non-trivial elements of the Shafarevich-Tate group
of an Abelian Variety over a Number Field}, Journal of Number Theory 133 (2013), no.6, 1977--1990.


\bibitem[Aga]{agashe:visfac}
A.~Agashe, \emph{A visible factor of the special {L}-value}, 
J. Reine Angew. Math. (Crelle's journal), Vol. 644 (2010), 159--187.

\bibitem[BLR90]{neronmodels}
S.~Bosch, W.~L{\"u}tkebohmert, and M.~Raynaud, \emph{N\'eron models},
  Springer-Verlag, Berlin, 1990.
  MR~\textbf{91i}:14034
  


\bibitem[Cre97]{cremona:algs}
J.\thinspace{}E. Cremona, \emph{Algorithms for modular elliptic curves}, second
  ed., Cambridge University Press, Cambridge, 1997.
  MR~\textbf{99e}:11068
  
\bibitem[CM00]{cremona-mazur}
J. E. Cremona and B.~Mazur, \emph{Visualizing elements in the
  {S}hafarevich-{T}ate group}, Experiment. Math. \textbf{9} (2000), no.~1,
  13--28. MR~1~758~797

\bibitem[Gon]{gonzalez}
Cristian D.~Gonzalez-Aviles, \emph{On Neron Class Groups of Abelian Varieties} 
(Preprint)

  




\bibitem[Maz72]{mazur:towers}
B.~Mazur, \emph{Rational points of abelian varieties with values in towers
  of number fields}, Invent. Math. \textbf{18} (1972), 183--266.
  MR~\textbf{56}:3020
  
 \bibitem[Mc90]{mccallum:wild}
W.~McCallum, \emph{Tate Duality and Wild Ramification}, Math. Ann. \textbf{288} (1990), 553--558


\bibitem[Mil80]{milne:etale}
J.~S. Milne, \emph{{\'E}tale Cohomology}, Princeton University Press, Princeton, 1980.

\bibitem[Mil86]{milne:duality}
J.~S. Milne, \emph{Arithmetic duality theorems},
Second Edition, Book-Surge Publishers, 2006.
    
  
\bibitem[Ser79]{serre:localfields}
J-P. Serre, \emph{Local fields}, Springer-Verlag, New York, 1979, Translated
from the French by Marvin Jay Greenberg.

\bibitem[Ser97]{serre:galois}
J-P. Serre, \emph{Galois Cohomology}, Springer-Verlag, 1997, Translated
from the French by Patrick Ion. (Corrected Second Printing, 2000)

\bibitem[Sil86]{silverman:AEC}
J.~Silverman, \emph{The Arithmetic of Elliptic Curves}, Springer, 1986.

\bibitem[Ste04]{stein:tamagawa}
W.\thinspace{}A. Stein, \emph{Tamagawa numbers \& Visibility}, AMS November
Meeting, Pittsburgh (2004) ; pdf available at\\ 
\url{http://modular.math.washington.edu/talks/pittsburgh/}

\bibitem[Tat57]{tate:wc}
J.~Tate, \emph{$WC$-groups over $p$-adic fields}, S{\'e}minaire N.~Bourbaki,
1956-1958, exp. no. 156, pp. 265--277.
  

\end{thebibliography}
\end{document}